\documentclass{amsart}
\numberwithin{equation}{section}

\newtheorem{thm}{Theorem}[section]
\newtheorem{lem}[thm]{Lemma}

\newtheorem{defin}[thm]{Definition}
\newtheorem{rem}[thm]{Remark}

\newcommand{\re}{\mathbb{R}}

\begin{document}
\title{Determination of the order of fractional derivative for
subdiffusion equation}

\author{Ravshan Ashurov}
\author{Sabir Umarov}
\address{Institute of Mathematics,
Uzbekistan Academy of Science, Tashkent, 81 Mirzo Ulugbek str.
100170}
\email{ashurovr@gmail.com}
\address{
University of New Haven, Department of Mathematics and Physics, 300 Boston Post Rd, West Haven, CT 06516}
 \email{sumarov@newhaven.edu}

\small

\title[Determination of the order of fractional derivative]
{Determination of the order of fractional derivative for
subdiffusion equation}
\vspace{2cm}

\begin{abstract}

The identification of the right order of the equation in applied fractional modeling plays an important role.
In this paper we consider an inverse problem for determining the order of time fractional
derivative in a subdiffusion equation with an arbitrary  second
order elliptic differential operator.  We prove that the additional information about the solution at a fixed time instant at a monitoring location, as "the observation
data", identifies uniquely the order of the fractional derivative.

\vskip 0.3cm \noindent {\it AMS 2000 Mathematics Subject
Classifications} :
Primary 35R11; Secondary 74S25.\\
{\it Key words}: Subdiffusion equation, Riemann-Liouville
derivatives, inverse and initial-boundary value problem,
determination of order of derivative, Fourier method.
\end{abstract}

\maketitle
\section{Introduction}

 It is well known (see, for example, \cite{Mach} - \cite{Zh1}) that Brownian motion,
 discovered in the first half of the 19th century, models motions of molecules in gases,
 electrons in semiconductors, neutrons in nuclear reactors, and much more.  The main difference between
 Brownian motion and processes obeying Newton's laws is that the diffusion
 packet spreads according to the law $t^{1 / 2}$ (and not like $t$).
 The subdiffusion process is characterized
 by a fractional exponent $\rho \in (0, 1)$, which is included in the diffusion
 equation as the order of the fractional time derivative.

 The theory of differential equations with fractional
 derivatives has gained considerable popularity and importance in the past
 few decades, mainly due to its applications in numerous seemingly distant
 fields of science and technology (see, for example, \cite{Mach} - \cite{Gor}).  In turn,
 the mathematical aspects of fractional differential equations and methods
 for solving them have been studied by many authors (see, for example, \cite{SU2} - \cite{LiLiu}).

 By inverse problems in the theory of partial differential equations are commonly
 called problems in which, together with solving a differential equation,
 it is also necessary to determine a coefficient(s) of the
 equation or/and the right side (source function).
 Naturally, in this case additional information should be given to find a new unknown function.
 Note that the interest in studying inverse problems
 for equations of mathematical physics is due to the importance of their
 applications in many branches of modern science, including mechanics, seismology, medical tomography, epidemics,
 and geophysics, just to mention a few. A significant number of studies are devoted to inverse problems of
 determining the right-hand side of subdiffusion equations
 (see, for example, \cite{Mach}, \cite{Zh2} - \cite{Mal} and references therein).

The present paper is devoted to the other important
type of inverse problems, namely to determining of the order of fractional
derivative in a subdiffusion equation, which is considered to
govern the anomaly of diffusion. More precisely, this inverse problems is
the determination of the unknown order of time-derivative in order
to match available data such as $u(x_0,t)$, $0<t<T,$ at a
monitoring point $x_0\in \Omega$. One of the practical example is a modeling of COVID-19 outbreak. The data \cite{JH} presented by Johns Hopkins University about the outbreak from different countries seem to show fractional order dynamical processes, in which the identification of fractional order rate of change is a key issue \cite{YQChen,Khan}. 

The problem of identification of fractional order of the model was considered by some researchers.
Note that all the publications assumed the fractional derivative of order $0<\rho< 1$ in the sense of Caputo
and studied mainly the uniqueness problem.
In paper \cite{Che} by J. Cheng et al.  an inverse
problem for determining the order of the Caputo fractional
derivative and a coefficient of one-dimensional time-fractional
diffusion equation is studied. The authors attached the homogeneous Neumann
boundary condition and the initial value given by the Dirac delta
function. They proved that the order of derivative and the unknown
coefficient are uniquely determined by the known datum $u(0,t), 0<t<T$. The
uniqueness of a solution of the two parameter inverse problem is
considered in paper  \cite{Tar} by Tatar and Ulusoy for the
differential equation 
\[
\partial_t^\rho u(t,x) = - (-\triangle)^\gamma u(t,x), \quad t>0, \ x \in \re^N.
\] 
The multi-term time-fractional diffusion equation and
distributed order fractional diffusion equations considered in
papers Li et al. \cite{LiY} and \cite{LiL}, correspondingly.

In paper \cite{Zhe} by X. Zheng et al.  the authors tried to
solve the most difficult problem of determining the variable order
of the Caputo fractional differentiation.  In this work, as in many
other papers, only the question of uniqueness is considered.  But
in our opinion, Lemma 4.1 of this paper is questionable, since there
exist functions (see, for example, \cite{AAP}) whose Fourier
series converge to zero in a certain region, but not all Fourier
coefficients are zero.

The following two papers \cite{Jan} and \cite{Hat} deals with the existence problem. J. Janno \cite{Jan} considered
a one-dimensional time-fractional diffusion equation with Caputo
derivatives. Giving an extra boundary condition $B u(\cdot, t)=
h(t), 0<t<T$ the author succeeded to prove the existence theorem
for determining the order of the derivative and the kernel of the
integral operator in the equation. The complexity of the proof of
the existence can be seen from the statement of corresponding
theorem (Theorem 7. 2 is formulated on more than one journal page).
In the paper of Hatano et al. \cite{Hat}  the
equation $\partial_t^\rho u =\Delta u,$ where $\Delta$ is the Laplace operator, is considered with the Dirichlet boundary
condition and the initial function $\varphi(x)$. They proved
that if $\varphi\in C_0^\infty(\Omega)$ and $\Delta
\varphi(x_0)\neq 0$, then
$$\rho =\lim\limits_{t\rightarrow
0}\big[t\partial_t u(x_0,t)[u(x_0,t)-\varphi(x_0)]^{-1}\big]$$.

In the recent survey paper \cite{LiLiu} by Z. Li et al.  in the section of
Open Problems they noted: "The studies on inverse
problems of the recovery of the fractional orders are far from
satisfactory since all the publications either assumed the
homogeneous boundary condition or studied this inverse problem by
the measurement in $t\in (0,\,\infty)$. It would be interesting to
investigate inverse problem by the value of the solution at a
fixed time as the observation data".

In the present work we address this problem. Namely, as
follows from our main result, in the case of the initial-boundary
(Neumann) value problem for the equation with the Riemann-Liouville
fractional derivative
\[
\partial^\rho_t u(x,t)=\Delta u(x,t), \quad x\in \Omega \subset \re^N, \  t>0,
\]
the only condition 
\begin{equation}
\label{t0} \int\limits_{\Omega} u(x,t_0) dx= d_0 \neq 0, 
\end{equation}
where $t_0 \ge 1$ is an observation time, recovers the order $\rho \in (0,1),$ and if we have two pairs of solutions
$\{u_1(x,t), \rho_1\}$ and $\{u_2(x,t), \rho_2\}$, then
$u_1(x,t)\equiv u_2(x,t)$ and $\rho_1=\rho_2$.

The paper is organized as follows. In the next section we formulate the main result. In Section \ref{section_fp} we prove the existence of a unique solution to the forward problem. This result will be used to prove the main result in Section \ref{ip}. Throughout the paper we assume that the fractional order $\rho$ of the main equation is constant. The solution of the forward problem is obtained under this assumption. The Cauchy problem in the case of piece-wise constant $\rho$ was studied in \cite{UmarovSteinberg} when $x \in \re^N.$  Modifying the result of this paper to the case of bounded domain $\Omega,$ one can extend the main result presented in this paper to the case of piece-wise constant $\rho,$ as well.


\section{Main result}
\label{section_mr}

\setcounter{section}{2} \setcounter{equation}{0}

Let $\Omega$ be an arbitrary $N$-dimensional domain with twice differentiable boundary $\partial\Omega.$  Namely, the
functions, defining the boundary equation in
local coordinates, are  continuously twice differentiable.
Let the second order differential operator
\[
A(x,D)u=\sum\limits_{i,j=1}^N D_i[a_{i,j}(x)D_ju]-c(x)u
\]
be a symmetric elliptic operator in $\Omega$, i.e.
\[
a_{i,j} (x) = a_{j,i}(x)\quad \text{and} \quad \sum_{i,j=1}^N
a_{i,j}(x) \xi_i \xi_j\geq a\sum_{i,j=1}^N  \xi^2_i,
\]
for all $x\in \Omega$ and  $\xi_i$, where $a=const
>0$ and $D_ju=\frac{\partial u}{\partial x_j}, \ j=1, \dots, N.$

Consider the spectral problem
\begin{equation}
\label{seq}
-A(x,D)v(x) =\lambda v(x), \quad x\in \Omega;
\end{equation}
\begin{equation}
\label{sbo}
Bv(x)\equiv\frac {\partial}{\partial n}v(x)+h(x)v(x)  = 0, \quad
x\in \partial \Omega,
\end{equation}
where ${n}$ is an external normal to the surface $\partial \Omega$.
It is known (see, for example, \cite{Il}, p. 100, \cite{Lad}, p. 111), that if $c(x)\geq 0$ and
\begin{align}
& a_{i, j}(x)\in C^{\lfloor \frac{N}{2}\rfloor +1}(\overline{\Omega}), \
i,j=1,\dots,N, \notag \\
&c(x)\in C^{\lfloor \frac{N}{2}\rfloor}(\overline{\Omega}), \quad h(x) \in
C^{\lfloor \frac{N}{2} \rfloor +2}(\partial \Omega), \label{coe}
\end{align}
where $\lfloor a \rfloor$ stands for the integer part of $a$, then the
corresponding inverse operator is compact, i.e. spectral problem
(\ref{seq}) - (\ref{sbo}) has a complete in $L_2(\Omega)$ set of
orthonormal eigenfunctions $\{v_k(x)\}$ and a countable set of
nonnegative eigenvalues $\{\lambda_k\}$.

The fractional part of our equation will be defined through  the Riemann-Liouville fractional derivative $\partial_t^\rho$ of order $0<\rho
< 1$. To define the Riemann-Liouville fractional derivative, one can define the fractional
integration of order $\rho<0$ of a function $f$ defined on $[0,
\infty)$ 
by the formula
\[
\partial_t^\rho f(t)=\frac{1}{\Gamma
(-\rho)}\int\limits_0^t\frac{f(\xi)}{(t-\xi)^{\rho+1}} d\xi, \quad
t>0,
\]
provided the right-hand side exists. Here $\Gamma(z)$ is
Euler's gamma function. Using this definition one can define the
Riemann-Liouville fractional derivative of order $\rho$,
$k-1<\rho\leq k$, $k\in \mathbb{N}$, as (see, for example,
\cite{SU2,PSK})
$$
\partial_t^\rho f(t)= \frac{d^k}{dt^k}\partial_t^{\rho-k} f(t).
$$
Note that if $\rho=k$, then fractional derivative coincides with the
ordinary classical derivative of order $k:$ $\partial_t^k f(t)
=\frac{d^k}{dt^k} f(t)$.

Let $0<\rho< 1$ be an unknown number to be determined.  Consider
the initial-boundary value problem
\begin{align}
\label{eq}
& \partial_t^\rho u(x,t) - A(x,D)u(x,t)  = 0, \quad x\in \Omega, \ 0<t\leq T,
\\
\label{bo}
& B u(x,t)\equiv\frac {\partial}{\partial n}u(x, t)+h(x)u(x, t) =0, \quad x\in \partial \Omega, \ t\geq 0,
\\
\label{in}
& \lim\limits_{t\rightarrow 0}\partial_t^{\rho-1} u(x,t) = \varphi(x), \quad x\in \overline{\Omega}.
\end{align}
Under some conditions on initial function $\varphi$ the solution
of this problem exists and is unique. This solution obviously
depends on $\rho$. The purpose of this paper is not only to find
the solution $u(x,t)$, but also to determine the order $\rho \in
(0, 1)$ of the time derivative. To do this one needs an extra
condition. As was mentioned above, different types of such
conditions were considered by a number of authors. 
We formulate our inverse problem in the following way. Let $\omega\in L_2(\Omega)$ be a
weight function with a property $||\omega||_{L_2(\Omega)}=1$. Is
it possible to determine the order of time fractional derivative
$0<\rho<1$ with the additional information
$$
\int\limits_{\Omega}u(x, t_0) \omega(x) dx=d_0,
$$
at a fixed time instant $t_0\geq 1$? This integral can be
considered as the average distribution of the quantity $u(t,x)$ over the
region $\Omega$ at the time instant $t=t_0$ with the weight
function $\omega$.

Below, using the classical Fourier method, we give
a positive answer to this question in the case when a weight
function is equal to the first eigenfunction of spectral problem
(\ref{seq}), (\ref{sbo}): $\omega(x) = v_1(x)$, i.e. when the
latter integral has the form

\begin{equation}\label{ex}
f(\rho; t_0)\equiv\int\limits_{\Omega}u(x, t_0) v_1(x) dx=d_0\neq
0.
\end{equation}
The quantity $f(\rho; t_0)$ is, in fact, the projection of the solution
$u(x,t_0)$ onto the first eigenfunction, and as is shown below, using the
Fourier method, under certain conditions,  $f(\rho; t_0)$ has a simple form $f(\rho; t_0)=d
t_0^{\rho-1}$, with a constant $d$ depending on $\rho$.

We call the initial-boundary value problem (\ref{eq}) - (\ref{in}) (for an arbitrary,  but fixed $0<\rho<1$)
the \textit{forward problem.} The initial-boundary value problem (\ref{eq})
- (\ref{in}) for an unknown $\rho,$ together with extra condition (\ref{ex}), is called an
\textit{inverse problem.}

\begin{defin} A pair $\{u(x,t), \rho\}$ of the function $u(x,t)$ and the parameter  $\rho$ with the properties
\begin{enumerate}
\item
$\rho \in (0, 1)$,
\item$\partial_t^\rho u(x,t), A(x, D)u(x,t)\in
C(\bar{\Omega}\times (0,\infty))$,
\item
$\partial_t^{\rho-1} u(x,t)\in C(\bar{\Omega}\times [0,\infty))$
\end{enumerate}
and satisfying all the conditions of
problem (\ref{eq}) - (\ref{ex}) in the classical sense is called
\textit{a classical solution} of inverse problem (\ref{eq}) -
(\ref{ex}).
\end{defin}

We will also call the classical solution simply the solution to
the inverse problem.
We draw attention to the fact, that in the definition of the
classical solution the requirement of continuity in the closed
domain of all derivatives included in equation (\ref{eq}) is not significant.
However, on the one hand, the uniqueness of such a solution is
proved quite simply, and on the other hand, the solution found by
the Fourier method satisfies the above conditions.

Further, let us denote by $\varphi_j$ the Fourier coefficients of the
function $\varphi(x)$ with respect to the system of eigenfunctions
$\{v_k(x)\}$, defined as a scalar product on $L_2(\Omega)$, i.e.
$\varphi_j=(\varphi, v_j)$. Let $D^{\alpha}$ stands for
$D_1^{\alpha_1}\cdots D_N^{\alpha_N}.$
For the seek of simplicity, we describe the proposed method, which is 
based on the classical Fourier method, for finding the order of
fractional differentiation in the case of $\lambda_1 = 0$ and
$\varphi_1\neq 0.$ Otherwise,  the method becomes
technically combersome.

Now we formulate the main result of this paper.


\begin{thm}
\label{thm_ip} Let the coefficients of operator $A(x, D)$ and function
$h(x)$ satisfy conditions (\ref{coe}) and let the first eigenvalue
of spectral problem (\ref{seq})-(\ref{sbo}) is equal to $0$:
$\lambda_1=0$ and $\varphi_1\neq 0$. Moreover, let the initial function
$\varphi(x)$ satisfy the conditions:
\begin{equation}\label{vp1}
\varphi(x)\in C^{\lfloor  \frac{N}{2} \rfloor }(\Omega),
\end{equation}
\begin{equation}\label{vp2}
D^\alpha \varphi(x)\in L_2(\Omega), \quad
|\alpha|=\big\lfloor \frac{N}{2} \big\rfloor +1,
\end{equation}
\begin{equation}\label{vp3}
B\varphi(x)=B A(x,D)\varphi(x)= \cdot\cdot\cdot\,
=BA^{\big\lfloor \frac{N}{4}\big\rfloor}(x,D)\varphi(x)=0, \quad x\in
\partial\Omega.
\end{equation}
Then inverse problem
(\ref{eq})- (\ref{ex}) has a unique solution $\{u(x,t), \rho \}$ if
and only if
\begin{equation}
\label{cond_100}
0<\frac{d_0}{\varphi_1} < 1.
\end{equation}
\end{thm}

\begin{rem} 
\begin{enumerate}
\item
Conditions (\ref{vp1})-(\ref{vp3}) are standard for
the existence of a solution to the forward problem (see for
example, \cite{Ruz}). 
Condition $\lambda_1=0$ of the
theorem is satisfied, for example, in the case of the Neumann
condition on the boundary for the Laplace operator.
\label{rem1}
\item
Theorem defines the unique $ \rho $ from (\ref {ex}). Hence, if we
define the integral (\ref {ex}) at another time instant $ t_1 $
and get a new $\rho_1$, i.e. $f(\rho_1; t_1)=d_1$, then from the
equality $f(\rho_1; t_0)=d_0$, by virtue of the theorem, we obtain
$\rho_1 = \rho$.
\end{enumerate}
\end{rem}

\section{Forward problem}
\label{section_fp}

\setcounter{section}{3} \setcounter{equation}{0} 

\begin{defin} A function $u(x,t)$ with the properties
\begin{enumerate}
\item
$\partial_t^\rho u(x,t), A(x,D)u(x,t)\in C(\bar{\Omega}\times
(0,\infty))$,
\item
$\partial_t^{\rho-1} u(x,t)\in C(\bar{\Omega}\times
[0,\infty))$
\end{enumerate}
and satisfying all the conditions of problem
(\ref{eq}) - (\ref{in}) in the classical sense is called
\textit{a classical solution} of forward problem (\ref{eq}) -
(\ref{in}).
\end{defin}

In this section we prove existence and uniqueness of the solution
of the forward problem by the Fourier method.
In accordance with the Fourier method, we will look for a solution to
problem  (\ref{eq}) - (\ref{in}) in the form of a series:
\begin{equation}\label{u}
u(x,t) = \sum\limits_{j=1}^\infty T_j(t) v_j(x), \quad t>0, \ x \in \Omega,
\end{equation}
where functions $T_j(t)$ are solutions to the Cauchy type problem
\begin{equation}\label{ca}
\partial_t^\rho T_j + \lambda_j T_j =0,\quad \lim\limits_{t\rightarrow 0}\partial_t^{\rho-1}
T_j(t)=\varphi_j, \quad \forall j=1,\dots, \infty.
\end{equation}
The unique solution of problem (\ref{ca}) has the form (see, for
example, \cite{PSK}, p. 16)
\begin{equation}\label{t}
T_j(t)=\varphi_j t^{\rho-1} E_{\rho, \rho} (-\lambda_j t^\rho),
\end{equation}
where $E_{\rho, \mu}$ is the Mittag-Leffler function
\[
E_{\rho, \mu}(t)= \sum\limits_{k=0}^\infty \frac{t^k}{\Gamma(\rho
k+\mu)}.
\]

\begin{thm}\label{dp} Let conditions (\ref{coe}) and (\ref{vp1}) - (\ref{vp3})
be satisfied. Then there exists a unique solution of the forward
problem (\ref{eq}) - (\ref{in}) and it has the representation
\begin{equation}\label{u1}
u(x,t)=\sum\limits_{j=1}^\infty\varphi_j t^{\rho-1} E_{\rho, \rho}
(-\lambda_j t^\rho) v_j(x),
\end{equation}
which absolutely and uniformly converges on $x\in \bar{\Omega}$
 for each $t\in (0, T]$.
\end{thm}

\begin{proof}
The uniqueness of the solution can be proved by the standard
technique based on completeness of the set of eigenfunctions
$\{v_k(x)\}$  in $L_2(\Omega)$ (see, for example, \cite{Ruz})

To prove the existence we need to introduce for any real number
$\tau$ an operator $\hat{A}^\tau$, acting in $L_2(\Omega)$ in the
following way
\[
\hat{A}^\tau g(x)= \sum\limits_{k=1}^\infty \lambda_k^\tau g_k
v_k(x), \quad g_k=(g,v_k).
\]
Obviously, the operator $\hat{A}^\tau$ with the domain of definition
\[
D(\hat{A}^\tau)=\{g\in L_2(\Omega):  \sum\limits_{k=1}^\infty
\lambda_k^{2\tau} |g_k|^2 < \infty\}
\]
is selfadjoint. If we denote by $A$ the operator in $L_2(\Omega)$,
acting as
\[
A v(x)=A(x,D) v (x)
\]
and with the domain of definition
\[
D(A)=\{v\in C^2(\bar{\Omega}): Bv(x)=0, \quad x\in \partial\Omega\},
\]
then it is not hard to show, that the operator $\hat{A}\equiv\hat{A}^1$
is the selfadjoint extension  of  the operator $A$ in $L_2(\Omega)$. In
the same way one can define the operator $(\hat{A}+I)^\tau,$ where $I$ is the identity operator in $L_2(\Omega).$

Further, we use the following lemma (see \cite{Kra}, p. 453):

\begin{lem}\label{CL} Let $\sigma > 1+\frac{N}{4}$. Then for any multi-index $\alpha$ satisfying $|\alpha|\leq 2$
the operator $D^\alpha (\hat{A}+I)^{-\sigma}$ (completely)
continuously maps the space $L_2(\Omega)$ into $C(\bar{\Omega}),$ and
moreover, the following estimate holds
\begin{equation}
\label{CL1} \|D^\alpha (\hat{A}+I)^{-\sigma} g\|_{C(\Omega)} \leq
C \|g\|_{L_2(\Omega)}.
\end{equation}
\end{lem}

Let $|\alpha|\leq 2$. First we prove that one can validly apply the
operators $D^\alpha$ and $\partial^\rho_t$ to the series
in (\ref{u1}) term-by-term.
Suppose that the  function $\varphi(x)$ satisfies the following condition
for some $\tau> \frac{N}{4}:$
\begin{equation}\label{vp}
\sum\limits_1^\infty (\lambda_j+1)^{2\tau} |\varphi_j|^2 \leq
C_\varphi<\infty.
\end{equation}
Consider the sum
\begin{equation}\label{S}
S_k(x,t)=\sum\limits_{j=1}^k v_j(x)\varphi_jt^{\rho-1}
E_{\rho,\rho}(-\lambda_j t^\rho).
\end{equation}
Since $(\hat{A}+I)^{-\tau-1} v_j(x) = (\lambda_j+1)^{-\tau-1}
v_j(x)$, we can rewrite the latter in the form
\[
S_k(x,t)=(\hat{A}+I)^{-\tau-1}\sum\limits_{j=1}^k
v_j(x)(\lambda_j+1)^{\tau+1}\varphi_jt^{\rho-1}
E_{\rho,\rho}(-\lambda_j t^\rho).
\]
Therefore, by virtue of Lemma \ref{CL}, one has
\begin{align} \notag
\left\|D^\alpha S^1_k \right\|_{C(\Omega)} & =\left\|D^\alpha
(\hat{A}+I)^{-\tau-1}\sum\limits_{j=1}^k
v_j(x)(\lambda_j+1)^{\tau+1}\varphi_jt^{\rho-1}
E_{\rho,\rho}(-\lambda_j t^\rho)\right\|_{C(\Omega)}
\\
\label{S1}
&\leq C \left\|\sum\limits_{j=1}^k
v_j(x)(\lambda_j+1)^{\tau+1}\varphi_jt^{\rho-1}
E_{\rho,\rho}(-\lambda_j t^\rho)\right\|_{L_2(\Omega)}.
\end{align}
Using the orthonormality of the system $\{v_j\}$, we have
\begin{equation}
\label{S2}
\left\|D^\alpha S_k\right\|^2_{C(\Omega)}\leq C \sum\limits_{j=1}^k
\left| (\lambda_j+1)^{\tau+1}\varphi_jt^{\rho-1}
E_{\rho,\rho}(-\lambda_j t^\rho) \right|^2.
\end{equation}
For the Mittag-Leffler function with a negative argument we have
an estimate  (see, for example, \cite{PSK},  p.13)
\[
\left|E_{\rho,\rho}(-t) \right| \leq \frac{C}{1+ t}, \quad t>0.
\]
Applying this inequality we have
\begin{align}
\sum\limits_{j=1}^k  & \left|(\lambda_j+1)^{\tau+1} \varphi_jt^{\rho-1}
E_{\rho,\rho}(-\lambda_j t^\rho) \right|^2                             
=\sum\limits_{\lambda_j<t^{-\rho}}
\left| (\lambda_j+1)^{\tau+1}\varphi_jt^{\rho-1}
E_{\rho,\rho}(-\lambda_j t^\rho) \right|^2                               \notag
\\ &
+
\sum\limits_{\lambda_j>t^{-\rho}}
\left|(\lambda_j+1)^{\tau+1}\varphi_jt^{\rho-1}
E_{\rho,\rho}(-\lambda_j t^\rho) \right|^2                              \notag
\\  & \leq
C t^{-2}(1+t^\rho)^2\sum\limits_{j=1}^k
(\lambda_j+1)^{2\tau}|\varphi_j|^2 \leq C
t^{-2}(1+T^\rho)^2C_\varphi.                                   \label{est_10}
\end{align}
Here we used the inequality $E_{\rho,\rho}(-\lambda_j t^\rho)<C$ in the case $\lambda_j<t^{-\rho},$ and the inequality 
 $E_{\rho,\rho}(-\lambda_j t^\rho)<C\frac{1}{\lambda_j t^\rho}$ in the case $\lambda_j>t^{-\rho}.$
Taking into account \eqref{est_10}, one can rewrite the estimate (\ref{S2}) as
\[
||D^\alpha S^1_k||^2_{C(\Omega)}\leq C
t^{-2}(1+T^\rho)^2C_\varphi.
\]
This implies uniform convergence on $x\in\bar{\Omega}$  of the
differentiated sum (\ref{S}) with respect to variables $x_j, \ j=1,\dots,N,$
for each $t\in (0,T]$. On the other hand, the sum (\ref{S1})
converges for any permutation of its members, as well, since these terms
are mutually orthogonal. This implies the absolute convergence of
the differentiated sum (\ref{S}) on the same interval $t\in
(0,T]$.

Further, it is not hard to see that
\begin{align}
\partial_t^\rho\sum\limits_1^k T_j(t)v_j(x) & =  -\sum\limits_1^k \lambda_j
T_j(t)v_j(x)                  \notag
\\
&=-A(x,D)(\hat{A}+1)^{-\tau-1}\sum\limits_1^k
(1+\lambda)^{\tau+1}_j T_j(t)v_j(x).             \notag
\end{align}
Absolute and uniform convergence of the latter series can be
proved as above.

Obviously, the function in (\ref{u1}) satisfies boundary conditions
(\ref{bo}). Considering the initial condition as (see, for
example, \cite{PSK} p. 104)
\begin{equation}\label{in1}
\lim\limits_{t\rightarrow
0}t^{1-\rho}u(x,t)=\frac{\varphi(x)}{\Gamma(\rho)},
\end{equation}
it is not hard to verify, that this condition is also satisfied.

Hence, if the function $\varphi(x)$ satisfies condition
(\ref{vp}), then all the statements of Theorem \ref{dp} hold.
As is shown in work \cite{Il} by V.A. Il'in (see also
\cite{Lad} p. 111) the fulfillment of conditions
(\ref{coe})-(\ref{vp2}) guarantee the convergence of the series in
(\ref{vp}). Thus, Theorem \ref{dp} is completely proved.
\end{proof}


\section{The inverse problem. Proof of Theorem \ref{thm_ip}}
\label{ip}

\setcounter{section}{4} \setcounter{equation}{0} 

In this section we prove the main result. First we prove the following auxiliary lemma:

\begin{lem}\label{f}
Let the first eigenvalue of the operator $A(x,D)$ iz zero and the Fourier coefficient $\varphi_1$ of $\varphi(x)$ is not zero. Then $f(\rho; t_0)$  as a function of $\rho \in (0,1)$ is strictly monotone: if $\varphi_1 > 0$, then $f(\rho;t_0)$ increases, and if
$\varphi_1 < 0$, then  $f(\rho; t_0)$ decreases. Moreover
\begin{equation}\label{>0}
\lim\limits_{\rho\rightarrow + 0} f(\rho; t_0)=0, \quad f(1;
t_0)=\varphi_1.
\end{equation}
\end{lem}

\begin{proof}
Since the system of eigenfunctions $\{v_j(x)\}$ are orthonormal
and $\lambda_1 =0$, then from (\ref{u1}) one has
\begin{equation} \label{frho_1}
f(\rho; t_0) = \varphi_1t_0^{\rho-1} E_{\rho,
\rho}(0)=\frac{\varphi_1 t_0^{\rho-1}}{\Gamma (\rho)}.
\end{equation}
Let $\Psi(\rho)$ be the logarithmic derivative of the gamma
function $\Gamma(\rho)$ (see, for example, \cite{Bat}). Then
$\Gamma'(\rho) = \Gamma (\rho) \Psi(\rho)$, and for $\rho \in
(0,1)$ we have $\Gamma(\rho)>0$ and $\Psi(\rho)<0$. Therefore,
$$
\frac{d}{d\rho}\bigg(\frac{t_0^{\rho-1}}{\Gamma(\rho)}\bigg)=\frac{t_0^{\rho-1}}{\Gamma(\rho)}\big[\ln
t_0- \Psi(\rho)\big] >0.
$$
for $t_0>1$ (the case $t_0=1$ is obvious). Thus function  $f(\rho;
t_0)$ increases or decreases depending on sign of $\varphi_1$. It
is easy to verify equalities (\ref{>0}).

\end{proof}

{\it Proof of Theorem \ref{thm_ip}}.  First we show existence of the
order of the fractional derivative $\rho$, which satisfies
condition (\ref{ex}). We have
$$
f(\rho; t_0)=\int\limits_\Omega u(x, t_0) v_1(x) dx = d_0.
$$
It follows from Lemma \ref{f} and representation \eqref{frho_1} immediately, that if
$$
0< \frac{d_0}{\varphi_1} < 1,
$$
then there exists a unique $\rho$, which satisfies condition
(\ref{ex}).

To prove the uniqueness of a solution of inverse problem
(\ref{eq})-(\ref{ex}) we suppose that there exist two pairs of
solutions $\{u_1, \rho_1\}$ and $\{u_2, \rho_2\}$ such, that
$0<\rho_k < 1, \ k=1,2,$ and
\begin{equation}\label{eq1}
\partial_t^{\rho_k} u_k(x,t) - A(x,D)u_{k}(x,t) = 0,  \quad k=1,2, \quad x\in \Omega, \ 0< t\leq T;
\end{equation}
\begin{equation}\label{in1}
\lim\limits_{t\rightarrow 0}\partial_t^{\rho_k-1} u_k(x,t) =
\varphi(x), \quad k=1,2, \quad x\in \overline{\Omega};
\end{equation}
\begin{equation}\label{bo1}
B u_k(x,t) =0, \quad k=1,2,  \quad x\in \partial\Omega, \ 0\leq t\leq T.
\end{equation}
Consider the following functions
$$
w^j_k(t)=\int\limits_\Omega u_k(x,t) v_j(x)dx, \quad k=1,2, \quad
j=1,2,\cdot\cdot\cdot
$$
Then, for each $j=1,2, \dots,$ we have
$$
\partial_t^{\rho_k} w^j_k(t) +\lambda_j w^j_k(t)=0, \quad \lim\limits_{t\rightarrow 0}\partial_t^{\rho_k-1} w^j_k(t) =
\varphi_j, \quad k=1,2.
$$
Therefore (see (\ref{t})), for each $j=1,2,\dots,$
$$
w_k^j(t)=\varphi_j t^{\rho_k-1}E_{\rho_k, \rho_k} (-\lambda_jt^{\rho_k}), \quad k=1,2,
$$
and condition (\ref{ex}) implies $w^1_1(t_0)=w^1_2(t_0)$. Since
$\lambda_1=0$ we obtain
$$
\varphi_1 t_0^{\rho_1-1}E_{\rho_1, \rho_1} (0)= \varphi_1
 t_0^{\rho_2-1}E_{\rho_2, \rho_2}(0)=\rho_0.
$$
As we have seen above (see Lemma \ref{f}), this equation yields
$\rho_1=\rho_2$. But in this case $w^j_1(t)=w^j_2(t)$ for all $t$
and $j$. Hence
$$
\int\limits_\Omega [u_1(x,t)-u_2(x,t)] v_j(x) dx=0
$$
for all $j$. Since the set of eigenfunctions $\{v_j\}$ is complete
in $L_2(\Omega)$, then we finally have $u_1(x,t)=u_2(x,t)$.
Thus, "if part" of the theorem is proved. 

To prove "only if" part of the theorem assume that condition \eqref{cond_100} is not verified. In this case, as it follows evidently from  representation \eqref{frho_1}, equation $f(\rho; t_0)=d_0$ has no solution on the interval $(0,1).$ Hence, in this case the inverse problem does not have a solution. The proof of Theorem \ref{thm_ip} is complete.

As an example of application of Theorem \ref{thm_ip} consider the following
initial-boundary value problem for one-dimensional diffusion
equation
\begin{equation}\label{eq1}
\partial_t^\rho u(x,t) - u_{xx}(x,t) = 0, \quad x\in (0, \pi), \quad t>0;
\end{equation}
with the  initial condition
\begin{equation}\label{in1}
\lim\limits_{t\rightarrow 0}\partial_t^{\rho-1} u(x,t) =
\varphi(x), \quad x\in [0, \pi];
\end{equation}
and the  boundary condition
\begin{equation}\label{bo1}
u_x(0,t) =0, \quad u_x(\pi,t) =0, \quad t\geq 0.
\end{equation}
where $0<\rho < 1$. In this case the
corresponding spectral problem has the
set of eigenfunctions $\{ \cos kx\}$ complete in $L_2(0, \pi)$, and eigenvalues $k^2$, $k=0,
1, \cdots.$ Note that the first eigenvalue in this case is $\lambda_1=0$ and the corresponding
eigenfunction is $v_1(x)=1$. Therefore, condition (\ref{ex}) takes the
form
\begin{equation}\label{ex1}
\int\limits_0^\pi u(x,t_0)dx =d_0, \quad t_0\geq 1.
\end{equation}

\begin{thm}\label{uniq} Let $\varphi\in C^1[0,\pi]$ and $\varphi_x(0)=\varphi_x(0)=0$. If
$$
0< \frac{d_0}{\varphi_1}<1,
$$
where $\varphi_1=\int\limits_0^{\pi} \varphi (x) dx,$ then there exists a unique solution $\{u(x,t), \rho\}$ to inverse
problem (\ref{eq1}) - (\ref{ex1}). Moreover, for the solution the
representation
$$
u(x,t)=\sum\limits_{j=1}^\infty\varphi_j t^{\rho-1} E_{\rho, \rho}
(-(j-1)^2 t^\rho) \cos (j-1) \, x \quad t>0, \ x \in [0,\pi],
$$
holds, where $\rho$ is the unique root of the algebraic equation
\[
\frac{t_0^{\rho-1}}{\Gamma(\rho)}= \frac{d_0}{\varphi_1}.
\]

\end{thm}
The proof of this theorem immediately follows  from Theorem
\ref{thm_ip}.

\bibliographystyle{amsplain}

\end{document}